\pdfoutput=1
\documentclass[reqno,oneside,letterpaper, 12pt]{amsart}
\usepackage{mathrsfs}
\usepackage[foot]{amsaddr}
\usepackage{amssymb}
\usepackage{wasysym}
\usepackage[utf8]{inputenc}
\usepackage{amsmath, amssymb,bm, cases, mathtools, thmtools}
\usepackage{verbatim}
\usepackage{graphicx}\graphicspath{{figures/}}
\usepackage{multicol}
\usepackage{tabularx}
\usepackage[usenames,dvipsnames]{xcolor}
\usepackage{mathrsfs}
\usepackage{url}
\usepackage[normalem]{ulem}
\usepackage{xstring}
\usepackage[shortlabels]{enumitem}

\usepackage[%
    minnames=1,maxnames=99,maxcitenames=3,
    style=authoryear,
    doi=false,url=false,
    firstinits=false,hyperref,natbib,backend=bibtex]{biblatex}
\renewbibmacro{in:}{%
  \ifentrytype{article}{}{\printtext{\bibstring{in}\intitlepunct}}}
\bibliography{biblio}

\usepackage[colorlinks,citecolor=blue,urlcolor=blue,linkcolor=RawSienna]{hyperref}
\usepackage{hypernat}
\usepackage{datetime}

\DeclareMathAlphabet\EuRoman{U}{eur}{m}{n}
\SetMathAlphabet\EuRoman{bold}{U}{eur}{b}{n}

\usepackage{euscript,microtype}
\usepackage[capitalize]{cleveref}

\crefname{assumption}{Assumption}{Assumptions}
\crefname{claim}{Claim}{Claims}

\makeatletter
\let\reftagform@=\tagform@
\def\tagform@#1{\maketag@@@{\ignorespaces\textcolor{gray}{(#1)}\unskip\@@italiccorr}}
\renewcommand{\eqref}[1]{\textup{\reftagform@{\ref{#1}}}}
\makeatother

\definecolor{WowColor}{rgb}{.75,0,.75}
\definecolor{SubtleColor}{rgb}{0,0,.50}

\newcounter{margincounter}

\declaretheorem[style=plain,numberwithin=section,name=Theorem]{theorem}
\declaretheorem[style=plain,sibling=theorem,name=Lemma]{lemma}

\declaretheorem[style=plain,sibling=theorem,name=Corollary]{corollary}
\declaretheorem[style=plain,sibling=theorem,name=Claim]{claim}

\declaretheorem[style=definition,sibling=theorem,name=Definition]{definition}
\declaretheorem[style=definition,sibling=theorem,name=Notation]{notation}

\crefformat{conditionINNER}{#2(#1)#3}
\crefmultiformat{conditionINNER}
  {(#2#1#3)}
  { and~(#2#1#3)}
  {, (#2#1#3)}
  { and~(#2#1#3)}
\Crefformat{conditionINNER}{Condition~#2(#1)#3}
\Crefmultiformat{conditionINNER}
  {Conditions~(#2#1#3)}
  { and~(#2#1#3)}
  {, (#2#1#3)}
  { and~(#2#1#3)}

\declaretheoremstyle[
    spaceabove=-6pt,
    spacebelow=6pt,
    headfont=\normalfont\bfseries,
    bodyfont = \normalfont,
    postheadspace=1em,
    qed=$\square$,
    headpunct={{}}]{myproofstyle}

\numberwithin{equation}{section}
\numberwithin{theorem}{section}

\usepackage{amssymb}
\usepackage{mathrsfs}
\usepackage{leftidx}

\usepackage{setspace}
\usepackage{thmtools}
\usepackage{geometry}
\usepackage{enumitem}
\usepackage{caption}
\usepackage{float}
\usepackage[colorinlistoftodos]{todonotes}
\usepackage[en-US]{datetime2}
\DTMlangsetup{showdayofweek=false} %

\def\[#1\]{\begin{align}#1\end{align}}
\def\*[#1\]{\begin{align*}#1\end{align*}}

\newcommand{\Reals}{\mathbb{R}}
\newcommand{\Nats}{\mathbb{N}}

\newcommand{\NNReals}{\Reals_{\ge 0}}

\DeclareMathOperator*{\newlim}{\mathrm{lim}\vphantom{\mathrm{infsup}}}

\DeclareMathOperator*{\newinf}{\mathrm{inf}\vphantom{\mathrm{infsup}}}
\DeclareMathOperator*{\newsup}{\mathrm{sup}\vphantom{\mathrm{infsup}}}
\renewcommand{\lim}{\newlim}

\renewcommand{\inf}{\newinf}
\renewcommand{\sup}{\newsup}

\newcommand{\cF}{\mathcal F}
\newcommand{\cG}{\mathcal G}

\newcommand{\NSE}[1]{{^{*}#1}}
\newcommand{\ST}{\mathsf{st}}

\newcommand{\PowerSet}{\mathscr{P}}

\newcommand{\cA}{\mathcal{A}}

\newtheorem{open problem}{Open Problem}

\newcommand{\Loeb}[1]{\overline{#1}}

\newcommand{\interior}[1]{%
  {\kern0pt#1}^{\mathrm{o}}%
}

\newcommand{\refproof}[1]{See \cref{#1} for \IfSubStr{#1}{,}{proofs}{a proof}. }

\newif\iflongform
\longformtrue

\iflongform

\else

\fi

\renewcommand\thmcontinues[1]{Continued}

\makeatletter
\providecommand*{\toclevel@definition}{0}
\providecommand*{\toclevel@theorem}{0}
\providecommand*{\toclevel@lemma}{0}
\makeatother
\usepackage{leftidx}

\makeatletter
\patchcmd{\@maketitle}
  {\ifx\@empty\@dedicatory}
  {\ifx\@empty\@date \else {\vskip3ex \centering\footnotesize\@date\par\vskip1ex}\fi
   \ifx\@empty\@dedicatory}
  {}{}
\patchcmd{\@maketitle}   %
  {\ifx\@empty\@date\else \@footnotetext{\@setdate}\fi}
  {}{}{}
\makeatother

{\color{blue}\title[]{Loeb Equivalence for General Internal Probability Spaces}}

\author[H.~Duanmu]{Haosui Duanmu$^{1}$}
\address{$^{2}$ Institute for Advanced Study in Mathematics, Harbin Institute of Technology}
\email{duanmuhaosui@hotmail.com}

\author[X.~Y.~Liu]{Xinyu Liu$^{1}$}
\email{liuxinyu@stu.hit.edu.cn}

\author[D.~Schrittesser]{David Schrittesser$^{1}$}
\email{david@logic.univie.ac.at}

\date{\today}

\newcommand{\cM}{\mathcal{M}}
\newcommand{\cN}{\mathcal{N}}
\newcommand{\cH}{\mathcal{H}}

\newtheorem{question}{Question}

\begin{document}
{ }\thanks{We gratefully acknowledge support from the National Science Foundation of China Grant No. 8220004225.}

\subjclass{28E05 (primary), 03H05, 26E35}

\maketitle

\begin{abstract}

Loeb measure theory, introduced in \citet{Loeb75}, stands as one of the most influential concepts in nonstandard analysis, underpinning nearly all applications in probability, stochastic processes, and mathematical economics. The paper resolves a fundamental open problem in Loeb measure theory originally posed by \citet{keisun}: let $(\Omega,\mathcal{F},\mu)$ and $(\Omega,\mathcal{G},\nu)$ be two Loeb equivalent internal probability spaces, and $\cH$ be the internal algebra generated from 
$\mathcal{F}\cup\mathcal{G}$. Does there exist an internal probability measure $P$ on $\cH$ such that $(\Omega,\mathcal{H},P)$ is Loeb equivalent to $(\Omega,\mathcal{F},\mu)$?
While \citet{loebeq23} recently provided a positive answer for hyperfinite probability spaces, the problem remained open for general internal probability spaces. We establish the existence of such an internal probability measure  for all internal probability spaces. 

\end{abstract}

\section{Introduction and Preliminaries}

Loeb measure theory is the cornerstone of applied nonstandard analysis. It has found applications in many areas in mathematics: for instance, in probability theory,\footnote{See for example \citet{andersonisrael}, \citet{localtime}, \citet{Keisler87}, \citet{brownfrac}, \citet{indmatching}, \citet{drw21}, and \citet{duffie25}} mathematical economics,\footnote{See for example \citet{nsexchange}, \citet{core1}, \citet{alipnas96}, \citet{sunlarge06}, \citet{anderson08}, and \citet{berknash24}} and statistics.\footnote{See for example \citet{nsbayes} and \citet{nscredible}} However, several fundamental problems in Loeb measure theory remain unresolved. 
\citet{keisun} posed four open problems on Loeb equivalence between internal probability spaces. 
Three of these have since been fully resolved by by \citet{loebeq21} and \citet{loebeq23}. This paper provides a complete solution to the final remaining problem.

We begin by introducing some fundamental concepts in Loeb measure theory. 
Given an internal probability space $(\Omega, \cF, \mu)$, its Loeb extension is defined to be the countably additive probability space $(\Omega, \bar{\cF}, \bar{\mu})$, where $\bar{\cF}$ consists of all sets $B \subseteq \Omega$ such that\footnote{In the following formula, we use $\ST$ to denote the standard part map.}
\begin{equation}
\sup\{\ST(\mu(A)) : B \supset A \in \cF\} = \inf\{\ST(\mu(C)) : B \subseteq C \in \cF\} \nonumber
\end{equation}
and $\bar{\mu}(B)$ is defined to be the above supremum.  \citet{keisun} introduced the following definition to compare two internal probability spaces.

\begin{definition}\label{def:loeb_equiv}
Let $\cM = (\Omega, \cF, \mu)$ and $\cN = (\Omega, \cG, \nu)$ be two internal probability spaces. We say $\cN$ \textbf{Loeb extends} $\cM$ if $\bar{\cG} \supset \bar{\cF}$ and $\bar{\nu}$ extends $\bar{\mu}$ as a function. We say $\cN$ is \textbf{Loeb equivalent} to $\cM$ if $\bar{\cF} = \bar{\cG}$ and $\bar{\nu} = \bar{\mu}$.
\end{definition}

Among the questions that \citet{keisun} proposed, the third one remains incompletely addressed in current research: 

\begin{question}\label{ksquestion}
 Suppose $\cM$ is Loeb equivalent to $\cN$, and let $\cH$ be the internal set algebra generated by $\cF \cup \cG$. Must there be an internal probability measure $P$ on $\cH$ such that $\cM$ is Loeb equivalent to $(\Omega, \cH, P)$?   
\end{question}

An internal probability space $(\Omega, \cF, \mu)$ is \emph{hyperfinite} if the internal algebra $\cF$ is hyperfinite.\footnote{Note that the internal sample space $\Omega$ is not necessarily hyperfinite.}
\citet{loebeq23} provide an affirmative answer to \cref{ksquestion} if $\cM$ and $\cN$ are hyperfinite. For general internal probability space, it remains unclear whether $\cH$ is a subset of $\Loeb{\cF}$, and consequently \cref{ksquestion} remains open. 
In this paper, we 
resolve \cref{ksquestion} in complete generality. Our proof proceeds as follows: 

In \cref{sec2}, discuss some basic facts of algebras of sets.
In particular, we provide an explicit normal form for elements in the generated algebra $\cH$, mirroring well-known arguments in measure theory or theory of Boolean algebras. 

In \cref{sec3}, we re-frame the problem in terms what we shall call \emph{essentially disjoint covers}, namely covers with sets whose pairwise intersections are contained in one fixed, global Loeb null set.
We show that such an essentially disjoint cover can indeed be found in several steps:
First, we establish the existence of what we call \emph{tight covers}.
We then show that the total size of the intersection errors can be estimated up to finite order by an intersection of finitely many (indeed, two) sets, each of which is a union of members of the cover. 
For a tight cover, the intersection of two such unions must be small; in other words, a tight cover must be an essentially disjoint cover. Together, this establishes the existence of an essentially disjoint cover.

This line of reasoning allows us to show $\cH \subseteq \bar{\cF}$. 
Finally, in \cref{sec4}, we resolve \cref{ksquestion} by a crucial application of a result by \citet{KR83}.

\section{Basic facts about algebras of sets}\label{sec2}

In this section, we show that 
any arbitrary element $H \in \cH$ can be written as the union of intersections of elements from $\cF$ and $\cG$. 
This canonical representation serves as the foundation for applying our subsequent approximation methods. 

We first take this opportunity to discuss the set of atoms in the algebra generated by a finite collection of sets. 
\begin{notation}\label{notation}
Suppose $\cF$ is an algebra of sets (a Boolean algebra) on the underlying set $\Omega$ and suppose $\{F_1, \hdots, F_k\} \subseteq \cF$ are given.
Then let us write,
for each set $S \subseteq \{1,\hdots, k\}$,
\[\nonumber
F_S = \begin{cases}
\bigcap_{i \leq k} F_i &\text{if $S = \Omega$}\\
\bigcap_{i \leq l} (\Omega\setminus F_i)  &\text{if $S = \emptyset$}\\
\left( \bigcap_{i \in S} F_i \right) \cap \left( \bigcap_{i \in \{1,\hdots, k\}\setminus S} (\Omega\setminus F_i) \right)&\text{if $S \notin \{\Omega,\emptyset\}$.}
\end{cases}
\]
Indeed, interpreting Boolean operations relative to $\Omega$ as we shall do throughout this paper, the last line without the qualification $S\notin \{\Omega,\emptyset\}$ suffices for this definition.
As is well-known, an equivalent definition is
    \[
\nonumber
F_S = \{ x \in \Omega \mid (\forall i \leq k)\;  x \in F_i \iff i \in S\}.
    \]
Moreover, let us use the short-hand
\[\nonumber
F_{-S} = F_{\{1,\hdots, k\}\setminus S}.
\]
It is also well-known that the non-empty members of $\big\{F_S \mid S \subseteq\{1, \hdots, k\}\big\}$ are precisely the atoms (the non-empty, $\subseteq$-minimal elements) of 
the subalgebra of $\cF$ generated by $\{F_1, \hdots, F_k\}$.

We need this notation so often that we issue the following convention: 
Whenever $\{X_1, \hdots, X_k\}$ is a family of sets in an algebra and $S\subseteq \{1, \hdots, k\}$,
we write $X_S$ for the set defined from $\{X_1, \hdots, X_k\}$ like $F_S$ was defined from $\{F_1, \hdots, F_k\}$ above (that is, reusing the same letter $X$ to indicate provenance of $X_S$ from $\{X_1, \hdots, X_k\}$).
\end{notation}

It should be clear that the sets $\big\{F_S \mid S \subseteq\{1, \hdots, k\}\big\}$ form a partition of $\Omega$; for the readers convenience, we include the straightforward verification of this fact.

\begin{claim}\label{particlaim}
The family $\{F_L : L\subseteq I\}$ forms a finite partition of $\Omega$ into mutually disjoint (not necessarily empty) sets.  
\end{claim}
\begin{proof}
It is clear that the family $\{F_L : L\subseteq I\}$ is finite. Let $L, L'\subset I$ be two different sets. Without loss of generality, there exists $i_0\in L$ but $i_0\not\in L'$. For any $x\in F_{L}$, we have $x\in F_{i_0}$. As $i_0\not\in L'$, we have $x\not\in F_{L'}$. Hence, we have $F_{L}\cap F_{L'}=\emptyset$. 

We now show that $\Omega=\bigcup_{L\subset I}F_{L}$. It is clear that $\bigcup_{L\subset I}F_{L}\subset \Omega$. Fix $x\in \Omega$. 
\begin{itemize}
    \item Suppose $x\in \Omega\setminus \bigcup_{i\leq k}F_i$. Then we have $x\in F_{\emptyset}$;
    \item Suppose $x\in \bigcup_{i\leq k}F_i$. Then there exists $L\subset I$ such that $x\in F_i$ if $i\in L$ and $x\not\in F_i$ if $i\not\in L$. 
\end{itemize}
Hence, we conclude that $\{F_L : L\subseteq I\}$ forms a finite partition of $\Omega$ into mutually disjoint (not necessarily empty) sets.  
\end{proof}

We next discuss a standard representation for the elements of the algebra generated by $\cF \cup \cG$, where $\cF$ and $\cG$ are algebras.

\begin{lemma}\label{lem:canonical_form}
Let $\cF$ and $\cG$ be algebras of sets on $\Omega$ and $\cH$ be the algebra generated by $\cF \cup \cG$. For every $H \in \cH$, there exists a sequence $\{(F_i, G_i)\}_{i=1}^k$ with $F_i \in \cF, G_i \in \cG$, and $k \in \Nats$, such that $H = \bigcup_{i=1}^k (F_i \cap G_i)$.
\end{lemma}
Note that the lemma makes no claim about elements $H$ of the $\sigma$-algebra generated by $\cF \cup \cG$; rather, our claim is about elements of the \emph{(finitely) generated algebra}. 
While an analogous statement does hold for the $\sigma$-algebra, this is of no use to us.

The proof of the lemma may be implicit in introductory texts on measure theory, probability, or Boolean algebras; but since the precise statement we need does not seem to be stated in any of these texts (that we know of) we decided to give a short proof for the convenience of the reader.

\begin{proof}
Define the collection $\mathcal{I} \triangleq \left\{ \bigcup_{i=1}^k (F_i \cap G_i) \mid k \in \Nats, \{F_i\}_{i\leq k} \subseteq \cF, \{G_i\}_{i \leq k} \subseteq \cG \right\}$. We show that $\mathcal{I}$ is an algebra.

\noindent \textbf{Closure under Finite Union:} Let $I_1, I_2 \in \mathcal{I}$. By definition, it is obvious that $I_1 \cup I_2 \in \mathcal{I}$.

\noindent \textbf{Closure under Complement:} Let $I \in \mathcal{I}$ with $I = \bigcup_{i=1}^k (F_i \cap G_i)$. By De Morgan's laws:
\begin{equation}
I^\complement = \left( \bigcup_{i=1}^k (F_i \cap G_i) \right)^\complement = \bigcap_{i=1}^k \left( F_i^\complement \cup G_i^\complement \right). \nonumber
\end{equation}
(Recall here that we consider Boolean operations such as complements to be relative to $\Omega$.)
Since $\cF$ and $\cG$ are algebras, let $A_i \triangleq F_i^\complement \in \cF$ and $B_i \triangleq G_i^\complement \in \cG$. The above expression becomes $I^\complement = \bigcap_{i=1}^k (A_i \cup B_i)$. Now:
\begin{equation}
\bigcap_{i=1}^k (A_i \cup B_i) = \bigcup_{S \in \PowerSet(\{1, 2, \dots, k\})} \left( \left(\bigcap_{i \in S} A_i \right) \cap \left(\bigcap_{j \in S^\complement} B_j \right) \right). \nonumber
\end{equation}
(E.g., to see an element $x$ of the left-hand side is also an element of the right-hand side, let $S=\{i \mid x \in A_i\}$.)
Invoking the conventions discussed in \cref{notation}, consider the sets $A_S \in \cF$ and $B_{-S} \in \cG$. 
The expression can be rewritten as:
\begin{equation}
I^\complement = \bigcup_{S \in \PowerSet(\{1, 2, \dots, k\})} (F_S \cap B_{-S}).\nonumber
\end{equation}
Thus, it is clear that $I^\complement \in \mathcal{I}$. 

Since $\mathcal{I}$ contains the empty set and is closed under finite unions and complements, $\mathcal{I}$ is an algebra. Because $\cH$ is the smallest algebra containing $\cF \cup \cG$, we conclude that $\cH \subseteq \mathcal{I}$ (of course it would not be hard to show that in fact, $\mathcal I = \mathcal H$). Thus, every $H \in \cH$ has the desired canonical form.
\end{proof}

We next verify that we can convert this canonical representation into a disjoint form, which is crucial for our later measure-theoretic approximations.

\begin{lemma}\label{lem:canonical_disjoint}
Let $\cF$ and $\cG$ be algebras of sets and $\cH$ be the algebra generated by $\cF \cup \cG$. For every $H \in \cH$, there exists a sequence $\{(F'_j, G'_j)\}_{j=1}^m$ where $F'_j \in \cF$, $G'_j \in \cG$, $m \in \Nats$, such that the sets $\{G'_j\}_{j=1}^m$ are pairwise disjoint and $H = \bigcup_{j=1}^m (F'_j \cap G'_j)$.
\end{lemma}

\begin{proof}
By Lemma \ref{lem:canonical_form}, any $H \in \cH$ can be written as $H = \bigcup_{i=1}^k (F_i \cap G_i)$ for some $k \in \Nats$. 

Again invoking \cref{notation},
enumerate $\{G_S \mid S\subseteq \{0,\hdots, k\}, \emptyset \neq G_S \subseteq \bigcup_{i=1}^k G_i\}$ as $\{G'_1, G'_2, \dots, G'_m\}$ for some $m \le 2^k - 1$. 
Thus, $\{G'_1, G'_2, \dots, G'_m\}$ is precisely the set of atoms of the algebra generated by $\{G_j\}_{j \leq k}$ which are subsets of $\bigcup_{j\leq k} G_i$ (see the discussion following \cref{notation}).

By construction, the collection $\{G'_1, G'_2, \dots, G'_m\}$ possesses three key properties:
\begin{enumerate}
\item $\bigcup_{j \leq m} G'_j = \bigcup_{i\leq k} G_i$
    \item The sets $G'_j$ are pairwise disjoint: $G'_j \cap G'_l = \emptyset$ for $j \neq l$.
    \item For any original set $G_i$ and any new set $G'_j$, either $G'_j \subseteq G_i$ or $G'_j \cap G_i = \emptyset$. 
\end{enumerate}
In particular, any original $G_i$ can be perfectly reconstructed as $G_i = \bigcup \{G'_j \mid G'_j \subseteq G_i\}$.

For each $j \in \{1, \dots, m\}$, we define 
\begin{equation}
F'_j = \bigcup \left\{ F_i \mid 1 \le i \le k, \, G'_j \subseteq G_i \right\}. \nonumber
\end{equation}
Clearly, $F'_j \in \cF$.

We now show that $\bigcup_{i=1}^k (F_i \cap G_i) = \bigcup_{j=1}^m (F'_j \cap G'_j)$.
\begin{itemize}
    \item $(\subseteq)$ Let $x \in \bigcup_{i=1}^k (F_i \cap G_i)$. Then there exists some index $i_0$ such that $x \in F_{i_0} \cap G_{i_0}$. Since $\{G'_j\}_{j=1}^m$ partitions $\bigcup_{i=1}^k G_i$, there exists a unique index $j_0$ such that $x \in G'_{j_0}$. Because $x \in G'_{j_0} \cap G_{i_0}$, the partition property dictates that $G'_{j_0} \subseteq G_{i_0}$. By the definition of $F'_{j_0}$, since $G'_{j_0} \subseteq G_{i_0}$, we must have $F_{i_0} \subseteq F'_{j_0}$. Thus, $x \in F'_{j_0} \cap G'_{j_0}$, which implies $x \in \bigcup_{j=1}^m (F'_j \cap G'_j)$.
    \item $(\supseteq)$ Let $x \in \bigcup_{j=1}^m (F'_j \cap G'_j)$. Then there exists some index $j_0$ such that $x \in F'_{j_0} \cap G'_{j_0}$. By definition, $x \in F'_{j_0}$ means there must exist some index $i_0$ such that $x \in F_{i_0}$ and $G'_{j_0} \subseteq G_{i_0}$. Since $x \in G'_{j_0}$ and $G'_{j_0} \subseteq G_{i_0}$, it naturally follows that $x \in G_{i_0}$. Therefore, $x \in F_{i_0} \cap G_{i_0}$, which implies $x \in \bigcup_{i=1}^k (F_i \cap G_i)$.
\end{itemize}
This confirms that the sequence $\{(F'_j, G'_j)\}_{j=1}^m$ satisfies all the required properties, yielding the desired disjoint representation for $H$.
\end{proof}

By transfer, we immediately obtain the following result, which we record for future reference. 

\begin{corollary}\label{cor:canonical_disjoint}
Let $\cM = (\Omega, \cF, \mu)$ and $\cN = (\Omega, \cG, \nu)$ be (not necessarily finite) internal probability spaces and $\cH$ be the internal algebra generated by $\cF \cup \cG$. For every $H \in \cH$, there exists a sequence $\{(F_i, G_i)\}_{i=1}^k$ where $F_i \in \cF$, $G_i \in \cG$, and $k \in \NSE{\Nats}$, such that the sets $\{G_i\}_{i=1}^k$ are pairwise disjoint and $H = \bigcup_{i=1}^k (F_i \cap G_i)$.
\end{corollary}

\section{Inner approximations from tight covers}\label{sec3}

Given $H \in \cH$, which we have established can be written as a disjoint union $\bigcup_{i \le k} (F_i \cap G_i)$, showing $H \in \bar{\cF}$ amounts to bounding $H$ from inside and outside by elements of $\cF$. 
Our strategy is to cover the disjoint components $G_i \in \cG$ using sets in $\cF$ which are still disjoint, provided one throws away a single fixed Loeb null set. 
The next lemma makes this idea precise by introducing what we call \emph{essentially disjoint covers} and \emph{inner approximations.}
We will show in \cref{thm:main_inclusion} that the existence of an inner approximation  quickly leads to $H \in \bar{\cF}$.

\begin{lemma}\label{lem:necessary_condition}
Suppose $\cM = (\Omega, \cF, \mu)$ and $\cN = (\Omega, \cG, \nu)$ are two Loeb equivalent general internal probability spaces. Fix a hyperfinite sequence of pairwise disjoint sets $\{G_i\}_{i \le k} \subseteq \cG$. If there exists an \emph{essentially disjoint cover}, that is, $\{F^{G_i}_{\textnormal{out}}\}_{i \le k}$ such that $F^{G_i}_{\textnormal{out}} \in \cF$ and $G_i \subseteq F^{G_i}_{\textnormal{out}}$ satisfying:
\begin{equation}
\mu \left( \bigcup_{i \neq j} \left( F^{G_i}_{\textnormal{out}} \cap F_{\textnormal{out}}^{G_j} \right) \right) \approx 0, \nonumber
\end{equation}
then, there exists an \emph{inner approximation}, that is, $\{F_{\textnormal{in}}^{G_i}\}_{i \le k}$ such that $F_{\textnormal{in}}^{G_i} \in \cF$ and $F_{\textnormal{in}}^{G_i} \subseteq G_i$ satisfying:
\begin{equation}
\nu \left( \bigcup_{i \le k} G_i \right) \approx \mu \left( \bigcup_{i \le k} F_{\textnormal{in}}^{G_i} \right). \nonumber
\end{equation}
\end{lemma}
\begin{proof}
Fix an essentially disjoint cover $\{F^{G_i}_{\textnormal{out}}\}_{i \le k}$ of $\{G_i\}_{i \le k} \subseteq \cG$ as in the theorem.
Let us write $G = \bigcup_{i\leq k} G_i$, 
$O = \bigcup_{i\neq j} (F^{G_i}_{\textnormal{out}} \cap F_{\textnormal{out}}^{G_j})$, and note
$
\mu(O) \approx 0$. 

Using Loeb equivalence, find $F \in \mathcal F$ 
and $\varepsilon  \approx 0$ 
such that $F \subseteq G$ and
\begin{equation} \nonumber
\nu(G) - \varepsilon \leq \mu(F).
\end{equation}
We claim that the following defines the desired inner approximation: Let
\[
\label{e.F_i}
F_{\textnormal{in}}^{G_i} = F \cap F^{G_i}_{\textnormal{out}} \setminus O.
\]
Since for any pair $i,j$ such that $j \neq i$, Eq. 
\eqref{e.F_i} implies
$F_{\textnormal{in}}^{G_i} \cap F_{\textnormal{out}}^{G_j} =\emptyset$ and hence
$F_{\textnormal{in}}^{G_i} \cap G_j =\emptyset$, and since clearly $F_{\textnormal{in}}^{G_i} \subseteq F \subseteq \bigcup_{j\leq k} G_j$, we see 
\begin{equation} \label{e:F_in}
F_{\textnormal{in}}^{G_i} \subseteq G_i.
\end{equation}
It is easy to see
\begin{equation} \nonumber
F \mathbin{\bigg\backslash} \bigcup_{i\leq k} F_{\textnormal{in}}^{G_i}  \subseteq O 
\end{equation}
(Suppose $x$ is an element of the left-hand-side. Since $x \in F \subseteq G$, there is $i \leq k$ such that $x \in F^{G_i}_{\textnormal{out}}$, and using $x \notin F_{\textnormal{in}}^{G_i}$ for this $i$, $x \in O$ follows from \eqref{e:F_in}.) It follows that
\[\nonumber
\mu(F) \leq \mu\left(\bigcup_{i\leq k} F_{\textnormal{in}}^{G_i}\right) + \mu(O).
\]
Putting it all together, we find
\begin{equation} \nonumber
\nu(G) -\varepsilon\leq  \mu(F) \leq \mu \left(\bigcup_{i \leq k} F_{\textnormal{in}}^{G_i} \right)  + \mu(O)\leq \nu(G) + \mu(O)
\end{equation}
showing that $\mu \left(\bigcup_{i \leq k} F_{\textnormal{in}}^{G_i}\right) \approx \nu(G)$.
\end{proof}

The following three lemmas lead to \cref{thm:main_contradiction}, establishing that if $\cM$ and $\cN$ are Loeb equivalent, an essentially disjoint cover
in the sense of \cref{lem:necessary_condition} is always attainable;
in other words, the hypothesis of \cref{lem:necessary_condition} always holds.
To this end, we first establish the existence of what we shall call a \emph{tight cover}.

\begin{lemma}
\label{lem:local tightness}
Suppose that $\mathcal M=(\Omega,\mathcal F,\mu)$ and $\mathcal N=(\Omega,\mathcal G,\nu)$ are two Loeb equivalent internal probability spaces. Let $\{G_i\}_{i\leq k}\subseteq\mathcal G$
be a hyperfinite sequence of pairwise disjoint sets. 
Then there exists a \emph{tight cover}, that is, a family $\{F^{G_i}_{\textnormal{out}}\}_{i\leq k}\subseteq\mathcal{F}$ such that
\begin{itemize}
    \item for every $i\leq k$, we have $G_i\subseteq F^{G_i}_{\textnormal{out}}$;
    \item Every $A \in \cF$ such that $A \subseteq \bigcup_{i\leq k}(F^{G_i}_{\textnormal{out}} \setminus
G_i)$ satisfies  $\mu(A) \approx 0$.
\end{itemize}
\end{lemma}

\begin{proof}
Let 
\[
\alpha=\inf\left\{\mu\left(\bigcup_{i\leq k}F_i\right) \;\Bigg{\vert}\; \{F_i\}_{i\leq k}\subseteq\mathcal{F}\text{ is internal}, G_i\subseteq F_i\ \text{for all $i\leq k$}\right\}. \nonumber
\]
By the definition of infimum, there exists an internal
hyperfinite sequence $\{F^{G_i}_{\textnormal{out}}\}_{i\leq k}\subseteq \mathcal{F}$ with $G_i\subseteq F^{G_i}_{\textnormal{out}}$ for each $i\leq k$ 
such that $\mu\left(\bigcup_{i\leq k} F^{G_i}_{\textnormal{out}}\right)<\alpha+\eta$ for some $\eta\approx 0$.
We claim that $\{F^{G_i}_{\textnormal{out}}\}_{i\leq k}$ is a tight cover.

We show that for an arbitrary set $A\in\mathcal F$ such that 
$A\subseteq
\bigcup_{i\leq k}(F^{G_i}_{\textnormal{out}} \setminus G_i)$ it must hold that $\mu(A) \leq \eta$. 
So let such a set $A$ be given and suppose towards a contradiction that $\mu(A) > \eta$. 
Define $(F_i^A)_{i\leq k}$ by 
\[\nonumber
F_i^A
=
F^{G_i}_{\textnormal{out}}\setminus A.
\]
Note that we have $A\cap G_i=\emptyset$ for $i\leq  k$. Thus, for $i\leq k$, we have $G_i\subseteq F^{G_i}_{\textnormal{out}}\setminus A=F_i^A$. Hence, 
we have $G_i\subseteq F_i^A$ for all $i\leq k$.
By construction, $\bigcup_{i\leq k}F^{G_i}_{\textnormal{out}} =  \left(\bigcup_{i\leq k}F^A_i\right) \cup A$ and the union is disjoint. Thus,
\[\nonumber
\begin{aligned}
     \mu\left(\bigcup_{i\leq k}F_i^A\right)
    =\mu\left(\bigcup_{i\leq k}   F^{G_i}_{\textnormal{out}}   \right) - \mu(A)
    < \alpha + \eta - \eta = \alpha
\end{aligned}
\]
which contradicts the choice of $\alpha$. 
Thus indeed $\mu(A)\leq \eta$. 
So 
$\{F^{G_i}_{\textnormal{out}}\}_{i\leq k}$ is a tight cover.
\end{proof}

The next lemma on standard measure spaces was proved with the assistance of a large language model.

\begin{lemma}\label{fincombprob}
    Let $(\Omega,\mathcal{F},\mu)$ be a measure space. Given a finite family $\{E_1,E_2,\ldots,E_k\}\subseteq\mathcal{F}$, define $O=\bigcup_{i\neq j}(E_i\cap E_j)$. Then, there exists $S\subseteq \{1,2,\ldots,k\}$ such that 
\[
\mu\left[\left(\bigcup_{i\in S}E_i\right)\cap\left(\bigcup_{i\notin S} E_i\right)\right]\geq \frac{1}{2}\mu(O). \label{E^S}
\]
\end{lemma}

The reader should take note that if $S= \emptyset$ or $S=\{1,\hdots, k\}$,
the set on the left-hand side of \cref{E^S} is empty.
So only in the trivial case that $\mu(O) = 0$ can the lemma produce $S$ of this form.

\begin{proof}
    Let \(I=\{1,\ldots,k\}\) be the index set. 
Once more, invoke \cref{notation} and recall that for $S \subseteq \{1, \hdots, k\}$,
\[\nonumber
E_S =  \{ x \in \Omega \mid (\forall i \leq k)\;  x \in E_i \iff i \in S\}.
\]

\begin{claim}\label{Orepresent}
$O=\bigcup_{\substack{L\subseteq I\\ |L|\ge2}}E_L$.  
\end{claim}
\begin{proof}
Let $x\in \bigcup_{\substack{L\subseteq I\\ |L|\ge2}}E_L$ be given. Then there exists some $L_1$ with $|L_1|\geq 2$ such that $x\in E_{L_1}$. Since $|L_1|\geq 2$, the set $L_1$ contains two distinct elements $i_1, j_1$. Then we have $x\in E_{i_1}\cap E_{j_1}\subseteq O$. So we have $\bigcup_{\substack{L\subseteq I\\ |L|\ge2}}E_L\subseteq O$.

Now let $x\in O$ be given. Then there exist two distinct $i_2, j_2\in I$ such that $x\in E_{i_2}\cap E_{j_2}$. Pick $L_2\subset I$ such that $i\in L_2$ if and only if $x\in E_i$. Note that $|L_2|\geq 2$ since $\{i_2, j_2\}\subset L_2$. We also have $x\in E_{L_2}$. Hence, we have $O\subseteq \bigcup_{\substack{L\subseteq I\\ |L|\ge2}}E_L$.
\end{proof}

For any given set $T\subseteq I$, let $O^T=(\bigcup_{i\in T}E_i)\cap(\bigcup_{i\notin T} E_i)$. For $T=\emptyset$ or $T=I$, this means $O^T=\emptyset$.

\begin{claim}\label{finprobclaim}
For every $L,T\subseteq I$,
\[\nonumber
E_L\cap O^T
=
\begin{cases}
E_L, & \text{if } L\cap T\neq\emptyset
        \text{ and } L\setminus T\neq\emptyset,\\[4pt]
\emptyset, & \text{otherwise}.
\end{cases}
\]
\end{claim}
\begin{proof}
Suppose $T=\emptyset$ or $T=I$.
For any $L\subset I$, 
we have either $L\cap T=\emptyset$ or $L\setminus T=\emptyset$, and $E_{L}\cap O^T=\emptyset$.  
We now focus on the case in which $T$ is a proper non-empty subset of $I$. Fix $x\in E_L$. We have
\[\nonumber
\begin{aligned}
x\in O^T
&\Longleftrightarrow
x\in \bigcup_{i\in T}E_i
\text{ and }
x\in \bigcup_{i\in I\setminus T}E_i\\
&\Longleftrightarrow
\text{there exists }i_0\in T\text{ such that }x\in E_{i_0}
\text{ and there exists }j_0\in I\setminus T\text{ such that }x\in E_{j_0}\\
&\Longleftrightarrow
\text{Since }x\in E_{L}, \text{ we have }
i_0 \in L
\text{ and }j_0\in L\\
&\Longleftrightarrow
L\cap T\neq\emptyset
\text{ and }
L\setminus T\neq\emptyset.
\end{aligned}
\]
Since our choice of $x$ is arbitrary, we have $E_L \subseteq O^T$ if and only if $L\cap T\neq\emptyset$ and $L\setminus T\neq\emptyset$. By the same argument, we have $E_L \cap O^T = \emptyset$ if and only if $L\cap T=\emptyset$ or $L\setminus T=\emptyset$.
\end{proof}

By \cref{particlaim},
the family $\{E_L:L\subseteq I\}$ forms a partition of $\Omega$. For a given $T\subseteq I$, by \cref{finprobclaim}, we have 
\[
O^T=O^T\cap(\bigcup_{L\subseteq I}E_L)=\bigcup_{L\subseteq I}(O^T\cap E_L)=\bigcup_{\substack{L\subseteq I\\ L\cap T \neq \emptyset\\L\setminus T\neq \emptyset}}E_L.\nonumber
\]
Since the family $\{E_L:L\subseteq I\}$ is mutually disjoint, we have 
\[
\mu(O^T) = \sum_{\substack{L \subseteq I \\ L \cap T \neq \emptyset \\ L \setminus T \neq \emptyset}} \mu(E_L).\nonumber
\]

For any fixed $L\subset I$, let $N_{L}$ be the number of sets $T\subset I$ such that $T\cap L\neq\emptyset$ and $L\setminus T\neq\emptyset$. Then we have\footnote{$\PowerSet(I)$ is the power set of $I$.}
\[
N_{L}&=|\PowerSet(I)\setminus \left(\{T\subset I: T\subset I\setminus L\}\cup \{T\subset I: L\subset T\}\right)| \nonumber \\
&=2^k-2^{k-|L|}-2^{k-|L|}=2^{k}(1-2^{1-|L|}). \nonumber
\]
We now calculate the average value of $\mu(O^T)$ where $T$ runs over all $T\subseteq I$.
\[
\frac{1}{2^k}\sum_{\substack{T \subseteq I }} \mu(O^T)
=\frac{1}{2^k}\sum_{\substack{T \subseteq I }}\sum_{\substack{L \subseteq I \\ L \cap T \neq \emptyset \\ L \setminus T \neq \emptyset}} \mu(E_L)= \frac{1}{2^k}\sum_{L \subseteq I}\sum_{\substack{T \subseteq I \\ L \cap T \neq \emptyset \\ L \setminus T \neq \emptyset}}  \mu(E_L)=\frac{1}{2^k}\sum_{L \subseteq I} N_{L}\,\mu(E_{L}). \nonumber
\]
Note that, if $|L|\leq 1$, then $L \cap T \neq  \emptyset$ and $L \setminus T \neq \emptyset$ cannot hold simultaneously and $N_L = 0$. On the other hand, for $L\subset I$ with $|L|\geq 2$, we have $N_{L}>0$. Hence, we have $N_{L}>0$ if and only if $|L|\geq 2$. 
Therefore
\[
\frac{1}{2^k}\sum_{\substack{T \subseteq I }} \mu(O^T)
&=\frac{1}{2^k}\sum_{\substack{L \subseteq I \\ |L|\geq 2}}2^k\bigl(1 - 2^{1-|L|}\bigr)\cdot \mu(E_L) \nonumber \\
&=\sum_{\substack{L \subseteq I \\ |L|\geq 2}}\bigl(1 - 2^{1-|L|}\bigr)\cdot \mu(E_L).\nonumber
\]
As $1 - 2^{1-|L|}\geq \frac{1}{2}$ for $|L|\geq 2$ and $O=\bigcup_{\substack{L\subseteq I\\ |L|\ge2}}E_L$ by \cref{Orepresent}, we have
\[\nonumber
\frac{1}{2^k}\sum_{\substack{T \subseteq I }} \mu(O^T)\geq\frac{1}{2}\mu(O).
\]
By this lower bound on the average of $\mu(O^T)$, clearly there must exist $S\subseteq I$ such that $\mu(O^S)\geq \frac{1}{2}\mu(O)$. 
\end{proof}

The transfer of \cref{fincombprob} leads to the following result:

\begin{corollary}\label{hyfinkcor}
 Let $(\Omega,\mathcal{F},\mu)$ be an internal probability space. Let $\{E_i\}_{i\le k} \subseteq \mathcal{F}$ be a hyperfinite family of internal sets. Let $O=\bigcup_{i\neq j}(E_i\cap E_j)$.
 Then there exists an internal set $S \subseteq \{1,\dots,k\}$ such that
\[
\mu\left( \bigcup_{i\in S} E_i \cap \bigcup_{i\notin S} E_i \right) \ge \frac{1}{2}\mu(O). \nonumber
\]
\end{corollary}

We now prove the following key result: Every hyperfinite family of pairwise disjoint measurable sets from one internal probability space can be covered by a hyperfinite family of essentially disjoint measurable sets from a Loeb equivalent internal probability space.   
Moreover, as we shall presently see, any tight cover is  already an essentially disjoint cover.

\begin{theorem}\label{thm:main_contradiction}
Let $\cM=(\Omega,\mathcal F,\mu), \cN=(\Omega,\mathcal G,\nu)$
be Loeb equivalent internal probability spaces. For every hyperfinite sequence of pairwise disjoint sets $\{G_i\}_{i\le k}\subseteq\mathcal G$ with $k\in \NSE{\Nats}$
there exists an essentially disjoint cover, that is, a family $\{F^{G_i}_{\textnormal{out}}\}_{i \le k} \subseteq\mathcal F$ of sets such that $G_i\subseteq F^{G_i}_{\textnormal{out}}$ for all $i\leq k$ and $\mu\left(\bigcup_{i\ne j}( F^{G_i}_{\textnormal{out}}\cap  F_{\textnormal{out}}^{G_j})\right)\approx0$.
\end{theorem}
\begin{proof}
By \cref{lem:local tightness}, there exists a tight cover, that is, an internal
hyperfinite family  $\{F^{G_i}_{\textnormal{out}}\}_{i\leq k}\subseteq\mathcal{F}$ such that
\begin{itemize}
    \item for every $i\leq k$, we have $G_i\subseteq F^{G_i}_{\textnormal{out}}$;
    \item Every $A \in \cF$ such that $A \subseteq \bigcup_{i\leq k}(F^{G_i}_{\textnormal{out}} \setminus
G_i)$ satisfies  $\mu(A) \approx 0$.
\end{itemize}
Define $O=\bigcup_{i\neq j}(F^{G_i}_{\textnormal{out}}\cap F_{\textnormal{out}}^{G_j})$. By \Cref{hyfinkcor}, there exists an internal set \(S\subseteq \{1,\ldots,k\}\) such
that $\mu(O^S)\geq\frac12\mu(O),$
where 
\[\nonumber
O^S
=
\left(\bigcup_{i\in S}F^{G_i}_{\textnormal{out}}\right)
\cap
\left(\bigcup_{i\notin S}F^{G_i}_{\textnormal{out}}\right).\]
Let $F^S=\bigcup_{i\in S}F^{G_i}_{\textnormal{out}}$, $F^{-S}=\bigcup_{i\notin S}F^{G_i}_{\textnormal{out}}$, $G^S=\bigcup_{i\in S}G_i$ and $G^{-S}=\bigcup_{i\notin S}G_i$. Thus $O^S=F^S\cap F^{-S}.$ 
Observe
\[\label{S}
O^S \setminus G^S \subseteq F^S \setminus G^S
= 
\left(\bigcup_{i\in S} F^{G_i}_{\textnormal{out}} \right)
\mathbin{\bigg{\backslash}} 
\left(\bigcup_{i\in S} G_i \right)
\subseteq 
\bigcup_{i\in S} (F^{G_i}_{\textnormal{out}} \setminus G_i)
\]
Since $O^S \subseteq F^{-S}$, a symmetric argument shows
\[\label{-S}
O^S \setminus G^{-S} \subseteq \bigcup_{i\notin S} (F^{G_i}_{\textnormal{out}} \setminus G_i).
\]
As $G^S$ and $G^{-S}$ are disjoint, we have $(G^S)^\complement \cup (G^{-S})^\complement = \Omega$ and
\[
O^S = (O^S \setminus G^S) \cup (O^S \setminus G^{-S}). \nonumber
\]
Together with \cref{S} and \cref{-S} this implies
\[\nonumber
O^S \subseteq \bigcup_{i\leq k} (F^{G_i}_{\textnormal{out}} \setminus G_i)
\]
Since  
$\{F^{G_i}_{\textnormal{out}}\}_{i\leq k}$
is tight cover and $O^S \in \cF$ we infer $\mu(O^S) \approx 0$.
Since $\mu(O_{S})\geq \frac{1}{2}\mu(O)$, we have $\mu\left(\bigcup_{i\ne j}( F^{G_i}_{\textnormal{out}}\cap  F_{\textnormal{out}}^{G_j})\right)=\mu(O)\approx 0$. 
So $\{F^{G_i}_{\textnormal{out}}\}_{i\leq k}\subseteq\mathcal{F}$ is the desired sequence.
\end{proof}

Putting together \Cref{thm:main_contradiction} and \Cref{lem:necessary_condition}, we have the following result:

\begin{corollary}\label{cor:inner_approx}
Let $\cM = (\Omega, \cF, \mu)$ and $\cN = (\Omega, \cG, \nu)$ be two Loeb equivalent general internal probability spaces. For any fixed hyperfinite sequence of pairwise disjoint sets $\{G_i\}_{i \le k} \subseteq \cG$, there exists an inner approximation $\{F_{\textnormal{in}}^{G_i}\}_{i \le k}$ such that $F_{\textnormal{in}}^{G_i} \in \cF$ and $F_{\textnormal{in}}^{G_i} \subseteq G_i$ satisfying
\begin{equation}
\nu \left( \bigcup_{i \le k} G_i \right) \approx \mu \left( \bigcup_{i \le k} F_{\textnormal{in}}^{G_i} \right).\nonumber
\end{equation}
\end{corollary}

With \cref{cor:inner_approx}, we now prove the main result of this section. 

\begin{theorem}\label{thm:main_inclusion}
Suppose $\cM = (\Omega, \cF, \mu)$ and $\cN = (\Omega, \cG, \nu)$ are two Loeb equivalent internal probability spaces and $\cH$ is the internal algebra generated by $\cF \cup \cG$. Then $\cH \subseteq \bar{\cF}$.
\end{theorem}

\begin{proof}
Let $H \in \cH$ be given.
Use Lemma \ref{lem:canonical_disjoint} to find a hyperfinite collection $\{(F_i, G_i) \mid i \le k\}$ such that the sets $G_i$ are pairwise disjoint and $H = \bigcup_{i \le k} (F_i \cap G_i)$. By Corollary \ref{cor:inner_approx}, there exists a family $\{F_{\textnormal{in}}^{G_i}\}_{i \le k} \subseteq \cF$, with $F_{\textnormal{in}}^{G_i} \subseteq G_i$, such that 
\[
\bar\nu \left( \bigcup_{i \leq k} G_i \right) = \bar\mu \left( \bigcup_{i\leq k} F_{\textnormal{in}}^{G_i}\right). \label{e:outer_approx}
\]
Let $H' = \bigcup_{i\leq k} F_i \cap F_{\textnormal{in}}^{G_i}$ and note that $H' \subseteq H$.
Moreover,
\begin{multline}
H\setminus H' 
= 
\left(\bigcup_{i\leq k} F_i \cap G_i\right)
\mathbin{\bigg\backslash}
\left(\bigcup_{i\leq k} F_i \cap F_{\textnormal{in}}^{G_i}\right)
=
\bigcup_{i\leq k} \left( F_i \cap \left(G_i \setminus F_{\textnormal{in}}^{G_i}\right) \right) \subseteq
 \\
\bigcup_{i\leq k} (G_i \setminus F_{\textnormal{in}}^{G_i} )
= 
\left(\bigcup_{i\leq k} G_i \right)
\mathbin{\bigg\backslash}
\left(\bigcup_{i\leq k} F_{\textnormal{in}}^{G_i} \right) \label{the_set}
\end{multline}
where the second equality holds because $F_{\textnormal{in}}^{G_i} \subseteq G_i$ and $\{G_i\}_{i\leq k}$ is a pairwise disjoint family (we leave the elementary verification to the reader). 
The last equality holds for the same reason. 
Clearly, \cref{e:outer_approx} implies that the last set of \cref{the_set} is Loeb null.
Therefore, as $H' \in \mathcal F$ and $H\setminus H'$ is Loeb null, we conclude $H \in \bar{\cF}$
\end{proof}

\section{Resolution of \cref{ksquestion}}\label{sec4}

In this section, we provide a full solution of \cref{ksquestion}. Following \citet{loebeq23}, we introduce the following definition:

\begin{definition}\label{aipdef}
Let $\Omega$ be an internal space equipped with an internal algebra $\cA$. An \emph{internal almost probability measure} $P$ is an internal mapping from $\cA$ to $\NSE{[0, 1]}$ such that
\begin{enumerate}
    \item $P(\emptyset)=0$ and $P(\Omega)=1$;
    \item For disjoint $A, B\in \cA$, we have $P(A\cup B)\approx P(A)+P(B)$. 
\end{enumerate}
\end{definition}

It is clear that every internal probability measure is an internal almost probability measure. In \citet{loebeq23}, we show that we can construct Loeb probability measures from internal almost probability measures:

\begin{lemma}[Lemma 2.2 in \citet{loebeq23}]\label{consloebaip}
Let $(\Omega, \cA, P)$ be an internal almost probability space. Then there exists a standard countably additive probability space $(X, \Loeb{\cA}, \Loeb{P})$ such that:
\begin{enumerate}
    \item $\Loeb{\cA}$ is a $\sigma$-algebra with $\cA\subseteq \Loeb{\cA}$;
    \item $\Loeb{P}(A)=\ST(P)$ on $\cA$;
    \item for every $A\in \Loeb{\cA}$ and every $\varepsilon>0$, there exist $A_i, A_o\in \cA$ such that $A_i\subset A\subset A_o$ and $P(A_o\setminus A_i)<\varepsilon$;
    \item for every $A\in \Loeb{\cA}$, there is a $B\in \cA$ such that $\Loeb{P}(A \triangle B)=0$.
\end{enumerate}
\end{lemma}

Since one can contruct a Loeb measure from an internal almost probability measure, for an internal almost probability space $(\Omega, \cA, P)$, with a slight abuse of notation, we call the standard probability space $(\Omega, \Loeb{\cA}, \Loeb{P})$ the \emph{Loeb space generated from} $(\Omega, \cA, P)$. The notions of Loeb extension and equivalence extend naturally to internal almost probability spaces.

\begin{definition}
Let $\cM = (\Omega, \cF, \mu)$ and $\cN = (\Omega, \cG, \nu)$ be two internal almost probability spaces. Then $\cN$ \emph{Loeb extends} $\cM$ if $\Loeb{\cF}\subseteq \Loeb{\cG}$ and $\Loeb{\nu}$ extends $\Loeb{\mu}$ as a function. Furthermore, $\cN$ is \emph{Loeb equivalent} to $\cM$ if $\Loeb{\cF}=\Loeb{\cG}$ and $\Loeb{\nu}=\Loeb{\mu}$.  
\end{definition}

The following theorem from \citet{loebeq23} provides a partial solution to \cref{ksquestion}. 

\begin{theorem}[Theorem 2.6 in \citet{loebeq23}]\label{thm:almost_measure}
Let $(\Omega, \cF, \mu)$ be an internal probability space and let $\cG$ be an internal algebra on $\Omega$. Let $\cH$ be an internal algebra generated by $\cF \cup \cG$. Then $(\Omega, \cH, P)$ is Loeb equivalent to $(\Omega, \cF, \mu)$ for some internal almost probability measure $P$ if and only if $\cH \subseteq \bar{\cF}$.
\end{theorem}

Thus, by \cref{thm:main_inclusion} and \cref{thm:almost_measure}, we immediately have: 

\begin{corollary}\label{aiplecor}
Suppose $\cM = (\Omega, \cF, \mu)$ and $\cN = (\Omega, \cG, \nu)$ are two Loeb equivalent general internal probability spaces and $\cH$ is the internal algebra generated by $\cF \cup \cG$. Then, there exists an internal almost probability measure $P$ such that $(\Omega, \cH, P)$ is Loeb equivalent to $(\Omega, \cF, \mu)$. 
\end{corollary}

Thus, to resolve \cref{ksquestion}, it is sufficient to show that there exists an internal probability measure $P'$ such that $(\Omega, \cH, P')$ is Loeb equivalent to $(\Omega, \cH, P)$.
To bridge the gap between almost additivity and exact strict additivity, We introduce the following definition and an important theorem by \citet{KR83}.

\begin{definition}[$\Delta$-approximately additive]\label{def:delta_add}
Let $\mathcal{A}$ be an algebra of sets. A set function $f : \mathcal{A} \to \Reals$ is called \textit{$\Delta$-approximately additive} for a constant $\Delta \ge 0$ if $f(\emptyset) = 0$ and for any disjoint sets $A, B \in \mathcal{A}$, we have $|f(A \cup B) - f(A) - f(B)| \le \Delta$.
\end{definition}

\begin{theorem}[\citet{KR83}]\label{thm:kr}
If $\mathcal{A}$ is an algebra of sets and $f : \mathcal{A} \to \Reals$ is a $\Delta$-approximately additive set function, then there exists a finitely additive set function $m : \mathcal{A} \to \Reals$ such that $|f(A) - m(A)| \le 45\Delta$ for all $A \in \mathcal{A}$.
\end{theorem}

We now show that the internal almost probability measure $P$ in \cref{aiplecor} can be canonically upgraded into an internal probability measure $P'$, hence resolving \cref{ksquestion}.

\begin{theorem}\label{thm:final_measure}
Suppose $\cM = (\Omega, \cF, \mu)$ and $\cN = (\Omega, \cG, \nu)$ are two Loeb equivalent general internal probability spaces and $\cH$ is the internal algebra generated by $\cF \cup \cG$. Then there exists an internal probability measure $P'$ such that $(\Omega, \cH, P')$ is Loeb equivalent to $\cM$.
\end{theorem}

\begin{proof}
By \cref{aiplecor}, there exists an internal almost probability measure $P$ such that $(\Omega, \cH, P)$ is Loeb equivalent to $\cM$.
Consider the set of all additivity errors of $P$:
\begin{equation}
E = \left\{ |P(A \cup B) - P(A) - P(B)| : A, B \in \cH, A \cap B = \emptyset \right\}.\nonumber
\end{equation}
Since $\cH$ and $P$ are internal, the internal definition principle implies that $E$ is an internal subset of $\NSE{\Reals}$. Since $P$ is an internal almost probability measure, we have $x\approx 0$ for all $x\in E$. Since $E$ is internal and bounded, we let 
$\delta = \sup E$. We claim that $\delta \approx 0$. If not, there would exist a standard real $\varepsilon > 0$ such that $\delta > \varepsilon$. Then there exists an $x \in E$ such that $x > \varepsilon$, which contradicts the fact that all elements of $E$ are infinitesimal. Hence, $\delta$ must be a non-negative infinitesimal in $\NSE{\Reals}$.

By the definition of $E$ and $\delta$, we have $|P(A \cup B) - P(A) - P(B)| \le \delta$ for all disjoint $A, B \in \cH$. Since $P(\emptyset) = 0$, $P$ is an internal $\delta$-approximately additive set function on the internal algebra $\cH$. By the Transfer Principle applied to Theorem \ref{thm:kr}, there exists an \textit{internal}, hyperfinitely additive set function $m : \cH \to \NSE{\Reals}$ such that
\begin{equation}\label{eq:6}
|m(A) - P(A)| \le 45\delta \approx 0, \quad \forall A \in \cH.
\end{equation}

Although $m$ is hyperfinitely additive, it may take negative values and we might have $m(\Omega)\neq 1$.  
Since $m$ is hyperfinitely additive, we can define its internal positive variation $m^+ : \cH \to \NSE{\NNReals}$ by:
\begin{equation}
m^+(A) = \sup \{ m(B) : B \subseteq A, B \in \cH \}.\nonumber
\end{equation}

Note that $m(\Omega) \ge P(\Omega) - 45\delta = 1 - 45\delta$. Since $\delta \approx 0$, we have $m(\Omega)>0$, and thus $m^+(\Omega) \ge m(\Omega) > 0$. Define $P' : \cH \to \NSE{\Reals}$ by:
\begin{equation}
P'(A) = \frac{m^+(A)}{m^+(\Omega)}, \quad \forall A \in \cH.\nonumber
\end{equation}
We show that $P'$ is an internal probability measure, and that $P'(A)\approx P(A)$ for $A\in \cH$. 

We first show that $P'$ is hyperfinitely additive. It suffices to show that the internal positive variation $m^+$ is hyperfinitely additive. By the internal induction principle, it suffices to show that, for two disjoint $A_1, A_2\in \cH$, we have $m^{+}(A_1 \cup A_2)=m^{+}(A_1)+m^{+}(A_2)$.
For any internal subset $B \subseteq A_1 \cup A_2$ with $B \in \cH$, we can partition $B$ into $B_1 = B \cap A_1$ and $B_2 = B \cap A_2$. Since $m$ is exactly finitely additive, we have $m(B) = m(B_1) + m(B_2)$.
Since $B_1 \subseteq A_1$ and $B_2 \subseteq A_2$, it follows from the definition of $m^+$ that $m(B_1) \le m^+(A_1)$ and $m(B_2) \le m^+(A_2)$. Thus, $m(B) \le m^+(A_1) + m^+(A_2)$. Taking the supremum yields $m^+(A_1 \cup A_2) \le m^+(A_1) + m^+(A_2)$.
Conversely, for any internal subsets $B_1 \subseteq A_1$ and $B_2 \subseteq A_2$ with $B_1, B_2\in \cH$, we have $B_{1}\cup B_{2}\subset A_1\cup A_2$. 
Therefore, $m^+(A_1 \cup A_2) \ge m(B_1 \cup B_2) = m(B_1) + m(B_2)$. Taking the supremum over $B_1$ and $B_2$ independently, we get $m^+(A_1 \cup A_2) \ge m^+(A_1) + m^+(A_2)$. Consequently, we have shown that $P'$ is hyperfinitely additive. 
From finite additivity we obtain $P'(\emptyset) = P'(\emptyset\cup\emptyset)= 2P'(\emptyset) = 0$;  moreover $P'(\Omega)=1$ by construction, so we conclude that $P'$ is an internal probability measure on $\cH$. 

It remains to show that $P'(A) \approx P(A)$ for all $A \in \cH$. We have\begin{align*}
    m^+(A) - m(A) &= \sup\{m(B)-m(A) : B \subseteq A, B \in \mathcal{H}\}\\& =\sup\{m(B)-m(B)-m(A\setminus B):B \subseteq A, B \in \mathcal{H}\}\\&=\sup\{-m(A\setminus B):B \subseteq A, B \in \mathcal{H}\}\\& \overset{A\setminus B=C}{=}\sup \{-m(C):C\subseteq A, C \in \mathcal{H}\}\\&=-\inf \{m(C):C\subseteq A, C \in \mathcal{H}\}
\end{align*} 
For any internal subset $C\in \cH$ with $C \subseteq A$, since $P$ maps $\cH$ to $\NSE{\NNReals}$ and $|m(A)-P(A)|\leq 45\delta$ for all $A\in \cH$,
we have:
\[
m(C) \ge P(C) - 45\delta \ge -45\delta.\nonumber
\]
Then $0 \le  m^+(A) - m(A) = -\inf \{m(C):C\subseteq A, C \in \mathcal{H}\} \le 45\delta$. 

Combining this with Equation (\ref{eq:6}), we obtain:
\[
|m^+(A) - P(A)| \le |m^+(A) - m(A)| + |m(A) - P(A)| \le 45\delta + 45\delta = 90\delta.\nonumber
\]
Since $\delta \approx 0$, we have $m^+(A) \approx P(A)$ for all $A \in \cH$. In particular, we have $m^+(\Omega) \approx P(\Omega) = 1$. Hence, we have:
\begin{equation}
P'(A) = \frac{m^+(A)}{m^+(\Omega)} \approx \frac{P(A)}{1} = P(A).\nonumber
\end{equation}

Since $P'(A) \approx P(A)$ for all $A \in \cH$, we have $\Loeb{P'}=\Loeb{P}$. 
Since $(\Omega, \cH, P)$ is Loeb equivalent to $(\Omega, \cF, \mu)$, $(\Omega, \cH, P')$ is also Loeb equivalent to $(\Omega, \cF, \mu)$, completing the proof.
\end{proof}

\printbibliography

\end{document}